\documentclass[11pt]{amsart}
\usepackage{amssymb,latexsym}
\usepackage{enumerate}
\usepackage[T1]{fontenc}
\usepackage{mathrsfs}
\usepackage{amsmath}
\usepackage{yhmath}
\makeatletter
\@namedef{subjclassname@2010}{%
  \textup{2010} Mathematics Subject Classification}
\makeatother
\makeindex

\usepackage{chngcntr}

\usepackage{mathtools}
\counterwithin{equation}{section}

\newtheorem{theorem}{Theorem}[subsection]
\newtheorem{cor}[theorem]{Corollary}

\newtheorem{prop}[theorem]{Proposition}

\theoremstyle{definition}

\newtheorem{sit}[theorem]{Situation}
\newtheorem{rem}[theorem]{Remark}

\numberwithin{theorem}{section}

\newcommand{\mb}{\mathbb}
\newcommand{\mc}{\mathcal}

\newcommand{\s}{\subset}

\begin{document}
\title[Peak functions and boundary behaviour...]{Peak functions and boundary behaviour of holomorphically invariant distances on strictly pseudoconvex domains}
\author{Arkadiusz Lewandowski}
\address{Institute of Mathematics\\ Faculty of Mathematics and Computer Science\\ Jagiellonian University\\ {\L}ojasiewicza 6,
30-348 Kraków, Poland}
\email{Arkadiusz.Lewandowski@im.uj.edu.pl}
\begin{abstract}
We give a parameter version of Graham-Kerzman approximation theorem for bounded holomorphic functions on strictly pseudoconvex domains. As an application, we present some uniform estimates for the boundary behaviour of the Kobayashi and Carath\'eodory pseudodistences on such domains. 
\end{abstract}

\subjclass[2010]{Primary 32T40; Secondary 32T15, 32F45}

\keywords{strictly pseudoconvex domains, peak functions, Kobayashi pseudodistance, Carath\'eodory pseudodistance}

\maketitle
\section{Introduction} For a bounded domain $G\s\mb{C}^n$, its boundary point $\zeta$ is called a \emph{peak point} with respect to $\mc{O}(\overline{G})$, the family of functions which are holomorphic in a neighborhood of $\overline{G},$ if there exist a function $f\in\mc{O}(\overline{G})$ such that $f(\zeta)=1$ and $f(\overline{G}\setminus\{\zeta\})\s\mb{D}:=\{z\in\mb{C}:|z|<1\}.$ Such a function is a \emph{peak function for $G$ at $\zeta$}. The peak functions turned out to be an important and fruitful concept in complex analysis, which has been used for instance to show the existence of (complete) proper holomorphic embeddings of strictly pseudoconvex domains into the unit ball $\mb{B}^N$ with large $N$ (see [\ref{For}],[\ref{DD}]), to estimate the boundary behavior of Carath\'{e}odory and Kobayashi metrics ([\ref{Bed}],[\ref{Gra}]), or to construct the solution operators for $\overline{\partial}$ problem with $L^{\infty}$ or H\"{o}lder estimates ([\ref{Forn}],[\ref{Ran2}]).\\
\indent It is well known that if $G$ is strictly pseudoconvex, then its every boundary point allows a peak function. It was Graham, who showed in [\ref{Gra}]  that in this situation there exists an open neighborhood $\widehat{G}$ of $G$, and a continuous function $h:\widehat{G}\times\partial G\rightarrow\mb{C}$ such that for $\zeta\in\partial G$, the function $h(\cdot;\zeta)$ is a peak function for $G$ at $\zeta$.\\
\indent Let us consider the following
\begin{sit}
Let $(G_t)_{t\in T}$ be a family of bounded strictly pseudoconvex domains with $\mc{C}^2$-smooth boundaries, where $T$ is a compact metric space with associated metric $d$. Suppose we have a domain $U\s\s\mb{C}^n$ such that
\begin{enumerate}[(i)]
\item $\displaystyle{\bigcup_{t\in T}\partial G_t\s\s U},$
\item for each $t\in T$ there exists a defining function $r_t\in\mc{C}^2(U)$ for $G_t$ such that its Levi form $\mc{L}_{r_t}$ is positive on $U\times(\mb{C}^n\setminus\{0\}),$
\item for any $\varepsilon>0$ there exists a $\delta>0$ such that for any $s,t\in T$ with $d(s,t)\leq \delta$ there is $\|r_t-r_s\|_{\mc{C}^2(U)}<\varepsilon$.\label{C2}
\end{enumerate}\label{Situation}
\end{sit}
\indent Recently we have proved the following parameter version of Graham's result (cf. [\ref{L}]):
\begin{theorem}
Let $(G_t)_{t\in T}$ be a family of strictly pseudoconvex domains as in Situation \ref{Situation}.
Then there exists an $\varepsilon>0$ such that for any $\eta_1<\varepsilon$ there exist an $\eta_2>0$ and positive constants $d_1,d_2$ such that for any $t\in T$ there exist a domain $\widehat{G_t}$ containing $\overline{G_t}$, and functions $h_t(\cdot;\zeta)\in\mc{O}(\widehat{G_t}),\zeta\in\partial G_t$ fulfilling the following conditions:
\begin{enumerate}
{\item[\emph{(a)}] $h_t(\zeta;\zeta)=1, |h_t(\cdot;\zeta)|<1$ on $\overline{G_t}\setminus\{\zeta\}$ (in particular, $h_t(\cdot;\zeta)$ is a peak function for $G_t$ at $\zeta$), \label{a}}
{\item[\emph{(b)}] $|1-h_t(z;\zeta)|\leq d_1\|z-\zeta\|, z\in\widehat{G_t}\cap\mb{B}(\zeta,\eta_2),$\label{b}}
{\item[\emph{(c)}] $|h_t(z;\zeta)|\leq d_2<1, z\in\overline{G_t},\|z-\zeta\|\geq\eta_1.$\label{c}}
\end{enumerate}
\label{Main}
\end{theorem}
\begin{rem}
The principal strength of Theorem \ref{Main} lies in the uniformity of the estimates given there: namely, all of the constants $\varepsilon, \eta_2, d_1, d_2$ can be chosen independently of $t$.
\end{rem}
\begin{rem}
The crucial point of the proof of Theorem \ref{Main} is the setting of certain continuously varying $\bar{\partial}$ problems on some domains $\widetilde{G_t}$ such that $\overline{G_t}\s\widetilde{G_t},t\in T$, and solving them in a subtle way, with uniform estimate $C$, given by Theorems V.2.7 and V.3.6 from [\ref{Ran}] and not depending on the domains $G_t$, to get the continuously varying solutions.\label{R1}
\end{rem} 
The technique mentioned in the above Remark, together with Theorem \ref{Main} itself, is also a vital ingredient of the proof of the first result given in the hereby paper. This is the following approximation result for bounded holomorphic functions defined near the boundary points of strictly pseudoconvex domains:
\begin{theorem}\label{Main2}
Let $(G_t)_{t\in T}$ be a family of strictly pseudoconvex domains as in Situation \ref{Situation}. Then there exist an  $R>0$ such that the set $G_t\cap\mb{B}(\zeta,R)$ is connected for any $t\in T,\zeta\in\partial G_t$ and for every such $R$ there exists a $\rho< R$ with the property that for any $\varepsilon>0$ and any $m\in\mb{N}$ there exists an $L=L(m,\varepsilon, R)>0$ with the property that for any $t\in T,\zeta_t\in\partial G_t, f_t\in\mc{H}^{\infty}(G_t\cap\mb{B}(\zeta_t,R)),$ and any system of pairwise different points $\mc{W}_{m,t}=\{w_1^t,\ldots,w_m^t\}\s G_t\cap\mb{B}(\zeta_t,\rho)$ there exist an $\hat{f}_t\in\mc{H}^{\infty}(G_t)$ such that
\begin{enumerate}[\indent\upshape(A)]
\item $D^{\alpha}\hat{f}_t(w_j^t)=D^{\alpha}f_t(w_j^t)$ for $|\alpha|\leq 1$ and $j=1,\ldots, m,$
\item There exists an $N=N(\varepsilon,R,\mc{W}_{m,t})$ such that $$\|\hat{f}_t\|_{G_t}\leq(L+N)\|f_t\|_{G_t\cap\mb{B}(\zeta_t,R)}.$$ Moreover, if $m=1$, then $N$ can be chosen to be zero,
\item $\|\hat{f}_t-f_t\|_{G_t\cap\mb{B}(\zeta_t,\rho)}<\varepsilon\|f\|_{G_t\cap\mb{B}(\zeta_t,R)}.$
\end{enumerate}
\end{theorem}
\begin{rem}
Notice that if $m=1$, then the estimate in (B) depends in fact only on $\varepsilon$ and $R$.
\end{rem}
\begin{rem}
Theorem \ref{Main2} is a two-directional amplification of Theorem 2 from [\ref{Gra}] (see also Theorem 19.1.3 in [\ref{JP}]). Firstly, it gives the independence of constants $R,\rho,$ and $L$ of parameter $t\in T$, and secondly - in the interpolation problem there is no limitation for the number of points chosen.
\end{rem}
With Theorems \ref{Main} and \ref{Main2} at hand, we are able to get some stability results for the boundary behaviour of Kobayashi and Carath\'eodory pseudodistances on strictly pseudoconvex domains. For arbitrary domain $G$ the Carath\'eodory pseudodistance is defined as
$$
\boldsymbol{c}_G(z,w):=\sup\{\boldsymbol{p}(0,f(w)):f\in\mc{O}(G,\mb{D}),f(z)=0\},\quad z,w\in G,
$$
while the Kobayashi pseudodistance may be expressed as
\begin{multline*}
\boldsymbol{k}_G(z,w):=\inf\big\{\sum_{j=1}^N\boldsymbol{p}(\xi_j,\zeta_j):N\in\mathbb{N},\xi_j,\zeta_j\in\mathbb{D},\text{\ and\ }\exists p_o,\ldots,p_N\in G:\\ p_0=z,p_N=w,\exists f_j\in\mathcal{O}(\mathbb{D},G):f_j(\xi_j)=p_{j-1},f_j(\zeta_j)=p_j,j=1,\ldots, N\big\}, \\ z,w\in G,
\end{multline*}
with $\boldsymbol{p}$ standing for the Poincar\'e distance on $\mb{D}$. Note that it always holds true that $\boldsymbol{c}_G\leq\boldsymbol{k}_G.$ For good exposition on the topic of these and other holomorphically contractible objects, we refer the Reader to the monograph [\ref{JP}].\\ 
\indent In [\ref{FR}], some upper and lower estimates for the boundary behaviour of the Kobayashi pseudodistance $\boldsymbol{k}_G$ on strictly pseudoconvex domain $G$ are given. It is showed there that in the situation just described, for any couple of distinct points $\zeta,\xi\in\partial G$ there exist constants $K$ and $C$ such that
$$
\boldsymbol{k}_{G}(z,w)\geq-\frac{1}{2}\log\text{\rm dist}(z,\partial G)-\frac{1}{2}\log\text{\rm dist}(w, \partial G)-K
$$
whenever $z,w\in G$ are such that $z$ is close to $\zeta$ and $w$ is close to $\xi$, and
\begin{multline*}
\boldsymbol{k}_{G}(z,w)\leq-\frac{1}{2}(\log\text{\rm dist}(z,\partial G)+\log\text{\rm dist}(w,\partial G))\\+\frac{1}{2}(\log(\text{\rm dist}(z,\partial G)+\|z-w\|)+\log(\text{\rm dist}(w,\partial G)+\|z-w\|))+C,
\end{multline*}
whenever $z,w\in G$ are close to $\zeta,$ cf. Corollary 2.4 and Proposition 2.5 in [\ref{FR}] (observe that for the upper estimate the strict pseudoconvexity is not needed - in [\ref{FR}] the domain $G$ is only assumed to have $\mc{C}^{1+\varepsilon}$ boundary). We prove that given $(G_t)_{t\in T}$, a family of strictly pseudoconvex domains as in Situation \ref{Situation}, the estimates as above are uniform with respect to $t\in T$ and $\zeta,\xi\in\partial G_t$, i.e. the bounds $K$ and $C$ given there can be taken independently of $t\in T$ and of $\zeta,\xi\in\partial G_t$ - in the first case depending only on $\|\zeta-\xi\|$ (with the domains $G_t$ not necessarily strictly pseudoconvex, when it comes to the upper estimate), see Propositions \ref{Prop1} and \ref{Prop2} below. These results are inspired by Propositions 9.1 and 9.2 from [\ref{MV}]. In correspondence to that paper, note that the role of the set of parameters $T$ is there played by a convergent sequence of numbers with its limit added. We also give some estimates in this spirit for the Carath\'eodory pseudodistence $\boldsymbol{c}_D$ - see Propositions \ref{Prop-1} and \ref{Prop0} (compare with Theorem 19.2.1, Corollary 19.2.2, and Proposition 19.2.4 from [\ref{JP}]).\\
\indent At the very end, as another corollary from Theorems \ref{Main} and \ref{Main2}, we deliver some uniform localization result for Carath\'eodory-Reiffen pseudometric - Proposition \ref{Reiffen} (this is parameter version of Proposition 6 from [\ref{Gra}], see also Theorem 19.3.1 in [\ref{JP}]).\\
\indent We close the Introducion with recalling that a bounded domain $G\s\mb{C}^n$ is called a \emph{strictly pseudoconvex} if there exist a neighborhood $U$ of $\partial G$ and a \emph{defining function} $r:U\rightarrow\mb{R}$ of class $\mc{C}^2$ on $U$ and such that
\begin{enumerate}
{\item[(I)] $G\cap U=\{z\in U:r(z)<0\}$,\label{Condition1}}
{\item[(II)] $(\mb{C}^n\setminus\overline{G})\cap U=\{z\in U:r(z)>0\},\label{Condition2}$}
{\item[(III)] $\nabla r(z)\neq 0$ for $z\in \partial G,$ where $\nabla r(z):=\left(\frac{\partial r}{\partial\overline{z}_1}(z),\cdots,\frac{\partial r}{\partial\overline{z}_n}(z)\right)$,\label{Condition3}}
\end{enumerate}
together with $$\mc{L}_r(z;X)>0\text{\ for\ } z\in\partial G\text{\ and\ nonzero\ }X\in T_z^{\mb{C}}(\partial D),$$  
where $\mc{L}_r$ denotes the Levi form of $r$ and $T_z^{\mb{C}}(\partial G)$ is the complex tangent space to $\partial G$ at $z$.\\ 
\indent It is known that $U$ and $r$ can be chosen to satisfy (I)-(III) and, additionally:
\begin{enumerate}
{\item[(IV)] $\mc{L}_r(z;X)>0$ for $z\in U$ and all nonzero $X\in\mb{C}^n,$\label{Condition4}}
\end{enumerate}
cf. [\ref{Kra1}].\\
\indent The proof of Theorem \ref{Main2} is presented in Section 2, while the uniform estimates for the boundary behaviour of Kobayashi and Carath\'eodory pseudodistances come in Section 3.
\section{Proof of Theorem \ref{Main2}}
\begin{proof}[Proof of Theorem \ref{Main2}]
Set $\eta_2<\eta_1,d_1,d_2<1,\widehat{G_t},$ and $h_t(\cdot;\zeta)$ for $t\in T,\zeta\in\partial G_t$ according to Theorem \ref{Main}, where $\eta_1$ is small enough to assure that the set $G_t\cap\mb{B}(\zeta,R)$ is connected for every $t\in T$ and $\zeta\in\partial G_t,$ where $R:=2\eta_1.$ Replacing $h_t$ with $\frac{h_t+3}{4}$ we may assume that $|h_t(z;\zeta)|\geq\frac{1}{2}, z\in\overline{G_t},\zeta\in\partial G_t.$\\
Let $d_3\in(d_2,1)$ and choose $0<\eta\leq\eta_2$ such that for any $t\in T$ we have $\mb{B}(\zeta;2\eta)\s\widehat{G_t}$ for all $\zeta\in\partial G_t$ as well as $|h_t(z;\zeta)|\geq d_3$ whenever $\zeta\in\partial G_t$ and $\|z-\zeta\|\leq\eta$ (this is possible because of the uniform choice of $d_1$ in theorem \ref{Main}). Define $\rho:=\min\{\frac{\eta}{2},\frac{\eta_1}{5}\}.$\\ 
For a fixed $t\in T$ there are points $\zeta_1^t,\ldots,\zeta_{N_t}^t\in\partial G_t$ such that $$\displaystyle{\partial G_t\s\bigcup_{j=1}^{N_t}\mb{B}(\zeta_j^t,\rho)}.$$\\
For any $j\in\{1,\ldots, N_t\}$ we modify the domain $G_t$ near the boundary point $\zeta_j^t$ in order to get a strictly pseudoconvex domain $G_j^t$ satisfying
\begin{enumerate}[\indent(1')]
\item $G_t\s G_j^t\s\widehat{G_t}\cap G_t^{(\eta)}$ (where $G_t^{(\eta)}$ denotes the $\eta$-hull of $G_t$)
\item $\overline{G_t}\cap\overline{\mb{B}(\zeta_j^t,2\rho)}\s\s G_j^t$ and dist$(\overline{G_t}\cap\overline{\mb{B}(\zeta_j^t,2\rho)},\partial G_j^t)\geq \beta>0$ with $\beta$ independent of $j$
\item $G_t\setminus\mb{B}(\zeta_j^t,\frac{7}{2}\rho)=G_j^t\setminus\mb{B}(\zeta_j^t,\frac{7}{2}\rho)$
\item The estimate $C$ for the solution of $\bar{\partial}$-problem for $G_t$ from Remark \ref{R1} is good for $G_j^t$. 
\end{enumerate}
Observe that for $s$ close enough to $t$ we may choose points $\zeta_1^s,\ldots,\zeta_{N_s}^s\in\partial G_s$ such that $N_s=N_t,$ $\zeta_j^s$ is close to $\zeta_j^t$ (with arbitrarily prescribed distance), $\displaystyle{\partial G_s\s\bigcup_{j=1}^{N_s}\mb{B}(\zeta_j^s,\rho)},$ and with the property that for any $j\in\{1,\ldots,N_s\}$ we can find strictly pseudoconvex deformation $G_j^s$ of $G_s$ near $\zeta_j^s$ such that
\begin{enumerate}
\item $G_s\s G_j^s\s\widehat{G_s}\cap G_s^{(\eta)}$ 
\item $\overline{G_s}\cap\overline{\mb{B}(\zeta_j^s,2\rho)}\s\s G_j^s$ and dist$(\overline{G_s}\cap\overline{\mb{B}(\zeta_j^s,2\rho)},\partial G_j^s)\geq \frac{\beta}{2}>0$
\item $G_s\setminus\mb{B}(\zeta_j^s,4\rho)=G_j^s\setminus\mb{B}(\zeta_j^s,4\rho)$
\item The estimate $C$ for the solution of $\bar{\partial}$-problem from Remark \ref{R1} is good for $G_j^s$. 
\end{enumerate}
Using the compactness of $T$, we see that the constants $C$ and $\beta$ do not depend on $t$.\\
Fix now $t=t_0\in T$ and $\zeta_0\in\partial G_t.$ Let $f\in\mc{H}^{\infty}(G_t\cap\mb{B}(\zeta_0,R))$ and take a system of pairwise different points $\mc{W}_{m,t}=\{w_1^t,\ldots,w_m^t\}\s G_t\cap\mb{B}(\zeta_0,\rho)$.\\
There exists a $j_0\in\{1,\ldots N_t\}$ such that $\zeta_0\in\mb{B}(\zeta_{j_0}^t,\rho)$. To simplify the notation, let us assume without loss of generality that $j_0=1.$\\
Choose a $\chi\in\mc{C}^{\infty}(\mb{C}^n,[0,1])$ such that $\chi\equiv 1$ on $\mb{B}(\zeta_0,\frac{6\eta_1}{5})$ and $\chi\equiv 0$ outside $\mb{B}(\zeta_0,\frac{9\eta_1}{5})$ and define $\alpha_t:=(\bar{\partial}\chi)f$ on $G_t\cap\mb{B}(\zeta_0,R)=G_t\cap\mb{B}(\zeta_0,2\eta_1)$ and $\alpha_t:=0$ on $G_t\setminus\mb{B}(\zeta_0,2\eta_1)$. Note that in view of the fact that $\alpha_t\equiv 0$ on $(G_t\cap\mb{B}(\zeta_0,\frac{6\eta_1}{5}))\cup(G_t\setminus\mb{B}(\zeta_0,\frac{9\eta_1}{5}))$, after trivial extension by zero, it can be treated as a $\bar{\partial}$-closed $(0,1)$-form of class $\mc{C}^{\infty}$ on $G_1^t$.\\
For $k\in\mb{N}$ (this will be specified later) consider the equation
\begin{equation}
\bar{\partial}v^t_k=(h_t(\cdot;\zeta_0))^k\alpha.\label{E1}
\end{equation}
The results mentioned in Remark \ref{R1} give a constant $C$, independent on $t$ and $j\in\{1,\ldots, N_t\}$, and a solution $v_k^t\in\mc{C}^{\infty}(G_1^t)$ of the problem (\ref{E1}) such that
$$
\|v_k^t\|_{G_1^t}\leq C\|(h_t(\cdot;\zeta_0))^k\|_{\text{spt}\alpha}\|\alpha\|_{G_1^t}.
$$
Further estimation gives
$$
\|v_k^t\|_{G_1^t}\leq CC_1 d_2^k\|f\|_{G_t\cap\mb{B}(\zeta_0,R)},
$$
with the constant $C_1$ depending only on $\eta_1$ (in particular, not depending on $t$).\\
Define the function $f_k:=\chi f-h_t(\cdot;\zeta_0)^{-k}v_k^t$ and observe it is holomorphic on $G_t$. Consequently, the function $h_t(\cdot;\zeta_0)^{-k}v_k^t$ is holomorphic on $G_1^t\cap\mb{B}(\zeta_0,\eta)$ (on the set $G_t\cap\mb{B}(\zeta_0,\eta)$ it follows from the holomorphicity of $f_k$, and on the remaining part - from the triviality of extension of $\alpha_t$ by zero and from the choice of $\eta$). Furthermore
$$
\|h_t(\cdot;\zeta_0)^{-k}v_k^t\|_{G_1^t\cap\mb{B}(\zeta_0,\eta)}\leq CC_1\left(\frac{d_2}{d_3}\right)^k\|f\|_{G_t\cap\mb{B}(\zeta_0,R)}.
$$
Note that for $z\in\overline{G_t}\cap\mb{B}(\zeta_0,\rho)$ we have $\|z-\zeta_1^t\|\leq \|z-\zeta_0\|+\|\zeta_0-\zeta_1^t\|\leq 2\rho.$ Therefore, $\overline{G_t}\cap\mb{B}(\zeta_0,\rho)\s\overline{G_t}\cap\mb{B}(\zeta_1^t,2\rho)\s\s G_1^t$.\\
On the set $G_t\cap\mb{B}(\zeta_0,\eta)$ we have the equality $f_k-\chi f=-h_t(\cdot;\zeta_0)^{-k}v_k,$ and the latter function is holomorphic on bigger set $G_1^t\cap\mb{B}(\zeta_0,\eta).$ Therefore, for $z\in G_t\cap\mb{B}(\zeta_0,\rho)$ we have
$$
\left|\frac{\partial f_k}{\partial z_j}(z)-\frac{\partial f}{\partial z_j}(z)\right|=\left|\frac{\partial}{\partial z_j}(h_t(\cdot;\zeta_0)^{-k}v_k)(z)\right|\leq\frac{CC_1}{L_1}\left(\frac{d_2}{d_3}\right)^k\|f\|_{G_t\cap\mb{B}(\zeta_0,R)},
$$ 
where the last inequality is a consequence of the Cauchy inequalities and in virtue of $(2)$ it may be chosen independently of $t,\zeta_0\in\partial G_t$, and $z\in G_t\cap\mb{B}(\zeta_0,\rho)$. The same argument gives
$$
\|f_k-f\|_{G_t\cap\mb{B}(\zeta_0,\eta)}\leq\frac{CC_1}{L_1}\left(\frac{d_2}{d_3}\right)^k\|f\|_{G_t\cap\mb{B}(\zeta_0,R)}.
$$
Set $\varepsilon>0.$ The remaining part of proof depends on $m$.\\
\textbf{Case 1.} $m=1$.\\
Put $w=w^t_1.$ Define $\hat{f}_k\in\mc{O}(G_t)$ by $\hat{f}_k(z):=f_k(z)+p(z),$ where
$$
p(z):=f(w)-f_k(w)+\sum_{j=1}^n\left(\frac{\partial f}{\partial z_j}(w)-\frac{\partial f_k}{\partial z_j}(w)\right)(z_j-w_j).
$$
It can be easily checked that $\hat{f}_k(w)=f(w),$ as well as $\frac{\partial \hat{f}_k}{\partial z_j}(w)=\frac{\partial f}{\partial z_j}(w).$ Furthermore,
\begin{multline*}
\|\hat{f}_k-f\|_{G_t\cap\mb{B}(\zeta_0,\rho)}\leq \|f_k-f\|_{G_t\cap\mb{B}(\zeta_0,\rho)}+|f(w)-f_k(w)|\\+n\text{diam}U\left|\frac{\partial f}{\partial z_j}(w)-\frac{\partial f_k}{\partial z_j}(w)\right|\leq (2+n\text{diam}U)\frac{CC_1}{L_1}\left(\frac{d_2}{d_3}\right)^k\|f\|_{G_t\cap\mb{B}(\zeta_0,R)}.
\end{multline*}
Let finally $k_0\in\mb{N}$ so large that $(2+n\text{diam}U)\frac{CC_1}{L_1}\left(\frac{d_2}{d_3}\right)^{k_0}\leq\varepsilon$ and define $\hat{f}:=\hat{f}_{k_0}.$ Observe that $k_0$ depends only on $\varepsilon$ and $\eta_1.$ It is left to estimate the norm of the latter function:
\begin{multline*}
\|\hat{f}\|_{G_t}\leq\|f_{k_0}\|_{G_t}+|f(w)-f_{k_0}(w)|+n\text{diam}U\left|\frac{\partial f}{\partial z_j}(w)-\frac{\partial f_{k_0}}{\partial z_j}(w)\right|\\\leq\|f_{k_0}\|_{G_t}+\varepsilon\|f\|_{G_t\cap\mb{B}(\zeta_0,R)}.
\end{multline*}
This, together with the estimate
\begin{equation}
\|f_{k_0}\|_{G_t}\leq\|\chi f\|_{G_t}+\|(h_t(\cdot;\zeta_0))^{-k_0}v_k^t\|_{G_t}\leq(1+2^{k_0}CC_1 d_2^{k_0})\|f\|_{G_t\cap\mb{B}(\zeta_0,R)}\label{E2}
\end{equation}
gives the conclusion with $L:=1+2^{k_0}CC_1 d_2^{k_0}+\varepsilon$ and $N=0.$\\
\textbf{Case 2.} $m\geq 2.$\\
We introduce some useful notation: for pairwise distinct complex numbers $w^1,\ldots, w^m$, the basis Lagrange polynomials are defined as
$$
l_i(z)=l_i^m(z):=\prod_{j=1,j\neq i}^m\frac{z-w^j}{w^i-w^j},\quad i=1,\ldots, m.
$$ 
Given a $z=(z_1,\ldots,z_n)$ and $z^1=(z_1^1,\ldots,z_n^1),\ldots,z^m=(z_1^m,\ldots,z^m_n)\in\mb{C}^n$ let us put
$$
l_{i,j}(z_j):=\prod_{k=1,k\neq i}^m\frac{z_j-w^k_j}{w^i_j-w^k_j},\quad i=1,\ldots,m,j=1,\ldots,n.
$$ 
Put $w^i:=w^t_i,i=1,\ldots, m.$ Assume first that $w^{i_1}_j\neq w^{i_2}_j$ for $j=1,\ldots, n$ whenever $i_1\neq i_2$. Define $\hat{f}_k\in\mc{O}(G_t)$ by $\hat{f}_k(z):=f_k(z)+p(z),$ where
\begin{multline*}
p(z):=\\\sum_{i=1}^m(f(w^i)-f_k(w^i))\left[\left(1+\frac{2}{n}\sum_{j=1}^n\frac{\partial l_{i,j}}{\partial z_j}(w^i_j)(w^i_j-z_j)\right)\frac{1}{n}\sum_{j=1}^n\left(l_{i,j}(z_j)\right)^2\right]\\
+\sum_{i=1}^m\left(\sum_{j=1}^n\left(\frac{\partial f}{\partial z_j}(w^i)-\frac{\partial f_k}{\partial z_j}(w^i)\right)(z_j-w^i_j)(l_{i,j}(z_j))^2\right).
\end{multline*}
One can verify that $\hat{f}_k(w^i)=f(w^i)$ and $\frac{\partial f_k}{\partial z_j}(w^i)=\frac{\partial f}{\partial z_j}(w^i)$
for $i=1,\ldots,m$ and $j=1,\ldots, n$. We estimate
\begin{multline*}\|\hat{f}_k-f\|_{G_t\cap\mb{B}(\zeta_0,\rho)}\\\leq \|f_k-f\|_{G_t\cap\mb{B}(\zeta_0,\rho)}+(M(1+2M\text{\rm diam}U))\sum_{i=1}^m|f(w^i)-f_k(w^i)|\\+M\text{\rm diam}U\sum_{i=1}^m\sum_{i=1}^n\left|\frac{\partial f}{\partial z_j}(w^i)-\frac{\partial f_k}{\partial z_j}(w^i)\right|\\\leq\frac{CC_1}{L_1}\left(\frac{d_2}{d_3}\right)^k(1+Mm(1+M(2+n)\text{\rm diam}U)\|f\|_{G_t\cap\mb{B}(\zeta_0,R)},
\end{multline*}
where $M=M(\mc{W}_{m,t})$. The last term is smaller than $\varepsilon\|f\|_{G_t\cap\mb{B}(\zeta_0,R)},$ provided that $k=k_0$ is sufficiently large (observe this choice of $k$ is independent of $f$). Then, performing similar computations, for $\hat{f}:=\hat{f}_{k_0}$, because of (\ref{E2}), we get
\begin{multline*}\|\hat{f}\|_{G_t}\leq\|{f_{k_0}}\|_{G_t}+\frac{CC_1}{L_1}\left(\frac{d_2}{d_3}\right)^{k_0}(Mm(1+M(2+n)\text{\rm diam}U)\|f\|_{G_t\cap\mb{B}(\zeta_0,R)}\\
\leq\left(1+2^{k_0}CC_1d_2^{k_0}+\frac{CC_1}{L_1}\left(\frac{d_2}{d_3}\right)^{k_0}\left(Mm(1+M(2+n)\text{\rm diam}U\right)\right)\|f\|_{G_t\cap\mb{B}(\zeta_0,R)}\\
=:(L+N)\|f\|_{G_t\cap\mb{B}(\zeta_0,R)}.
\end{multline*}
Let us pass to the remaining case, namely: the one where we assume that some of points $w^1,\ldots, w^m$ have at least one common coordinate. For $i\in\{1,\ldots, m\}, j\in\{1,\ldots, n\}$ and $z=(z_1,\ldots, z_n)\in\mb{C}^n$ define
$$
\tilde{l}_{i,j}(z_j)
\prod_{k=1, k\neq i, w^k_j\neq w^i_j}^m\frac{z_j-w^k_j}{w^i_j-w_j^k} \quad \text{\ if\ }\exists k\neq i:w^i_j\neq w^k_j,
$$ 
and 
$$
\tilde{l}_{i,j}(z_j)\equiv 0 \quad \text{otherwise}.
$$
Observe that for a fixed $i$ not all of $\tilde{l}_{i,j}$ are zero (since the points $w^1,\ldots, w^m$ are pairwise different). Therefore, we can define
$$
A_i:=\{j\in\{1,\ldots, n\}:\tilde{l}_{i,j} \text{\ is\ nonzero}\}\neq\varnothing,\quad n_i:=|A_i|,\quad 1=1,\ldots, m.
$$
Observe that if $j\notin A_i$, then 
$$B_{i,j}:=\{(k_i,m_i)\in\{1,\ldots,m\}\times(\{1,\ldots,n\}\setminus \{j\}):w^i_{m_i}\neq w^{k_i}_{m_i}\}\neq\varnothing.$$
Define
\begin{multline*}
p(z):=\\\sum_{i=1}^m(f(w^i)-f_k(w^i))\left[\left(1+\frac{2}{n_i}\sum_{j\in A_i}\frac{\partial \tilde{l}_{i,j}}{\partial z_j}(w^i_j)(w^i_j-z_j)\right)\frac{1}{n_i}\sum_{j\in A_i}\left(\tilde{l}_{i,j}(z_j)\right)^2\right]\\
+\sum_{i=1}^m\left(\sum_{j\in A_i}\left(\frac{\partial f}{\partial z_j}(w^i)-\frac{\partial f_k}{\partial z_j}(w^i)\right)(z_j-w^i_j)(\tilde{l}_{i,j}(z_j))^2\right)\\
+\sum_{i=1}^m\left(\sum_{j\notin A_i}\left(\frac{\partial f}{\partial z_j}(w^i)-\frac{\partial f_k}{\partial z_j}(w^i)\right)(z_j-w^i_j)\prod_{(k_i,m_i)\in B_{i,j}}\left(\frac{z_{m_i}-w_{m_i}^{k_i}}{w^i_{m_i}-w_{m_i}^{k_i}}\right)^2\right)
\end{multline*}
and $\hat{f}(z)=\hat{f}_k(z):=f(z)+p(z)$ with $k$ sufficiently large, and we end the proof carrying out similar computations as before.
\end{proof}
\begin{rem}
In Theorem \ref{Main2} one can require in the conclusion that $\|\hat{f}_t-f_t\|_{G_t\cap\mb{B}(\zeta_t,\rho)}<\varepsilon$. Analyzing the proof of our result, we see that it is possible to get this kind of estimate. There is, however, a price we have to pay - the constants $L,N$ are not any more independent of $f$.
\end{rem}
\begin{rem}
In [\ref{Gra}] it is stated that the estimate in (C) holds true also for the derivatives of $\hat{f}$ and $f$ up to previously prescribed order (eventually with interpolation at one given point). The same result is possible to get here - with $L$ still independent of $t.$
\end{rem}
\section{Boundary behaviour of Kobayashi and Carath\'eodory pseudodistances}
As indicated in the Introducion, few of the results to be presented are corollaries from Theorems \ref{Main} and \ref{Main2}. We start with the uniform estimates for the boundary behaviour of the Carath\'eodory pseudodistance.
\begin{prop}
Let $(G_t)_{t\in T}$ be a family of strictly pseudoconvex domains as in Situation \ref{Situation}. Let $K\s\mb{C}^n$ be a compact set such that $K\s G_t$ for any $t\in T.$ Then there exists a constant $C>0$ such that for every $t\in T$
$$
\boldsymbol{c}_{G_t}(z,w)\geq-\frac{1}{2}\log\text{\rm dist}(z,\partial G_t)-C
$$
whenever $z\in K,w\in G_t.$\\
In particular, if $(G_j)_{j\in\mb{N}}$ is a sequence of strictly pseudoconvex domains with $\mc{C}^2$-smooth boundaries such that $G_j$ converges to $G$, a strictly pseudoconvex domain with $\mc{C}^2$-smooth boundary with respect to $\mc{C}^2$ topology on domains, then for every compact $K\s G$ there exist a constant $C_1>0$ such that for sufficiently large $j$
$$
\boldsymbol{c}_{G_j}(z,w)\geq-\frac{1}{2}\log\text{\rm dist}(z,\partial G_j)-C_1
$$ 
whenever $z\in K,w\in G_j.$\label{Prop-1}
\end{prop}
\begin{proof}
This is a consequence of Theorem \ref{Main}. Observe that the proof may be carried out along the lines of the proof of Theorem 19.2.1 from [\ref{JP}]. Note that thanks to Theorem \ref{Main} the constants $\eta_1,\eta_2,d_1,$ and $d_2$ may be chosen independently of $t\in T$. Also, utilizing the compactness of $T$ together with (iii) from Situation \ref{Situation}, we see that it is possible to get an $\varepsilon_0$ in the proof of Theorem 19.2.1 from [\ref{JP}] in such a way, that it does not depend on $t$, as well.
\end{proof}
\begin{cor}\label{Cor1}
Let $(G_t)_{t\in T}$ be a family of strictly pseudoconvex domains as in Situation \ref{Situation}. Let $\varepsilon>0.$ Then there exist positive constants $\rho_2<\rho_1<\varepsilon$ and $C>0$ such that for every $t\in T$ and every $\zeta\in\partial G_t$ we have
$$
\boldsymbol{c}_{G_t}(z,w)\geq-\frac{1}{2}\log\text{\rm dist}(z,\partial G_t)-C
$$
whenever $z\in G_t\cap\mb{B}(\zeta,\rho_2)$ and $w\in G_t\setminus\mb{B}(\zeta,\rho_1).$\\
In particular, if $(G_j)_{j\in\mb{N}}$ is a sequence of strictly pseudoconvex domains with $\mc{C}^2$-smooth boundaries such that $G_j$ converges to $G$, a strictly pseudoconvex domain with $\mc{C}^2$-smooth boundary with respect to $\mc{C}^2$ topology on domains, and if $\varepsilon>0$ is given, then there exist positive constants $\rho_2<\rho_1<\varepsilon$ and $C_1>0$ such that for every $j\in\mb{N}$ and every $\zeta\in\partial G_j$ we have
$$
\boldsymbol{c}_{G_j}(z,w)\geq-\frac{1}{2}\log\text{\rm dist}(z,\partial G_j)-C_1
$$
whenever $z\in G_j\cap\mb{B}(\zeta,\rho_2)$ and $w\in G_j\setminus\mb{B}(\zeta,\rho_1).$
\end{cor}
To get the upper estimate, we turn to the Kobayashi pseudodistance case. 
\begin{prop}
Let $(G_t)_{t\in T}$ be a family of bounded domains with $\mc{C}^2$-smooth boundaries, where $T$ is a compact metric space with associated metric $d$. Suppose we have a domain $U\s\s\mb{C}^n$ such that $\displaystyle{\bigcup_{t\in T}\partial G_t\s\s U}$ and with the property that for any $\varepsilon>0$ there exists a $\delta>0$ such that for any $s,t\in T$ with $d(s,t)\leq \delta$ there is $\|r_t-r_s\|_{\mc{C}^2(U)}<\varepsilon,$  where $r_t$ denotes a defining function for $G_t$, defined on $U$ for any $t\in T.$ Let $K\s\mb{C}^n$ be a compact set such that $K\s G_t$ for any $t\in T.$ Then there exists a constant $C>0$ such that for any $t\in T$ we have
$$
\boldsymbol{k}_{G_t}(z,w)\leq-\frac{1}{2}\log\text{\rm dist}(w,\partial G_t)+C,
$$
whenever $z\in K$ and $w\in G_t.$\\
In particular, if $(G_j)_{j\in\mb{N}}$ is a sequence of bounded domains with $\mc{C}^2$-smooth boundaries such that $G_j$ converges to $G$, a bounded domain with $\mc{C}^2$-smooth boundary with respect to $\mc{C}^2$ topology on domains, then for every compact $K\s G$ there exist a constant $C_1>0$ such that for sufficiently large $j$
$$
\boldsymbol{k}_{G_j}(z,w)\leq-\frac{1}{2}\log\text{\rm dist}(w,\partial G_j)+C_1,
$$
whenever $z\in K$ and $w\in G_j.$\label{Prop0}
\end{prop}
\begin{rem}
In comparison with Situation \ref{Situation}, here we do not assume strict pseudoconvexity of the domains. Only the boundary regularity is important here.
\end{rem} 
\begin{proof}
The proof goes along the lines of the proof of Proposition 19.2.4 from [\ref{JP}]. We only contain here the necessary modifications.\\
Observe that $\varepsilon_0$ as in the proof of the mentioned result may be taken to be independent of $t\in T$.\\
For fixed $t\in T$ let $\delta>0$ be such that for any $s\in T$ with $d(s,t)\leq \delta$ the set
$$\displaystyle{
K_t:=\overline{\bigcup_{s\in T:d(s,t)\leq \delta}\{z\in G_s:\text{dist}(z,\partial G_s)\geq\varepsilon_0\}}
}$$
is compact in $G_s$ and, moreover, $\text{dist}(K_t,\partial G_s)\geq\frac{\varepsilon_0}{2}.$ Let $G_{t,\delta}$ be a bounded domain with $\mc{C}^2$-smooth boundary such that
$$\displaystyle{
K\cup K_t\s\s G_{t,\delta}\s\bigcap_{s\in T:d(s,t)\leq \delta}G_s.
}$$
For $s$ as above and $z\in K, w\in G_s$, with $\text{dist}(w,\partial G_s)\leq\varepsilon_0$ using the same argument as in [\ref{JP}], we get the estimate
$$
\boldsymbol{k}_{G_s}(z,w)\leq-\frac{1}{2}\log\text{dist}(w,\partial G_s)+\frac{1}{2}\log(2\varepsilon_0)+C_{t,\delta}
$$
with $C_{t,\delta}:=\sup\{\boldsymbol{k}_{G_{t,\delta}}(a,b):a,b\in K\cup K_t\}$. By the compactness of $T$, the latter constant may be chosen independently of $t$. We end the proof as in [\ref{JP}].
\end{proof}
The next result gives uniform lower estimate for the Kobayashi pseudodistance as the arguments approach two different boundary points of the domains.
\begin{prop}
Let $(G_t)_{t\in T}$ be a family of strictly pseudoconvex domains as in Situation \ref{Situation}. Let $\varepsilon>0$. Then there exists a constant $C>0$ such that for every $t\in T$ and every $\zeta,\xi,$ different points from $\partial G_t$ such that $\|\zeta-\xi\|\geq\varepsilon$, we have
$$
\boldsymbol{k}_{G_t}(z,w)\geq-\frac{1}{2}\log\text{\rm dist}(z,\partial G_t)-\frac{1}{2}\log\text{\rm dist}(w, \partial G_t)-C
$$
whenever $z,w\in G_t$ are such that $z$ is close to $\zeta$ and $w$ is close to $\xi$ (with uniform size of the respective neighborhoods).\\
In particular, if $(G_j)_{j\in\mb{N}}$ is a sequence of strictly pseudoconvex domains with $\mc{C}^2$-smooth boundaries such that $G_j$ converges to $G$, a strictly pseudoconvex domain with $\mc{C}^2$-smooth boundary with respect to $\mc{C}^2$ topology on domains, and if $\varepsilon>0$ is given, then there exists a constant $C>0$ such that for every $j\in\mb{N}$ and every $\zeta,\xi,$ different points from $\partial G_j$ or from $\partial G$ such that $\|\zeta-\xi\|\geq\varepsilon$, we have
$$
\boldsymbol{k}_{G_j}(z,w)\geq-\frac{1}{2}\log\text{\rm dist}(z,\partial G_j)-\frac{1}{2}\log\text{\rm dist}(w, \partial G_j)-C
$$
whenever $z,w\in G_j$ are such that $z$ is close to $\zeta$ and $w$ is close to $\xi$ (with uniform size of the respective neighborhoods). \label{Prop1}
\end{prop}
\begin{proof}
The proof goes similarly to the proof of Proposition 19.2.7 from [\ref{JP}] (see also [\ref{FR}]). Only, one has to use our Corollary \ref{Cor1} instead of Theorem19.2.2 from [\ref{JP}].
\end{proof}
\begin{rem}
Observe that Proposition 9.1 from [\ref{MV}] can be deduced from Proposition \ref{Prop1}.
\end{rem}
If the arguments approach the same boundary point, we have the following uniform estimate:
\begin{prop}
Let $(G_t)_{t\in T}$ be a family of bounded domains with $\mc{C}^2$-smooth boundaries, where $T$ is a compact metric space with associated metric $d$. Suppose we have a domain $U\s\s\mb{C}^n$ such that $\displaystyle{\bigcup_{t\in T}\partial G_t\s\s U}$ and with the property that for any $\varepsilon>0$ there exists a $\delta>0$ such that for any $s,t\in T$ with $d(s,t)\leq \delta$ there is $\|r_t-r_s\|_{\mc{C}^2(U)}<\varepsilon,$  where $r_t$ denotes a defining function for $G_t$, defined on $U$ for any $t\in T.$ Then there exists a constant $C>0$ such that for any $t\in T$ and any $\zeta\in G_t$ there exists a neighborhood $V=V(\zeta)$, of uniform size, with the property that
\begin{multline*}
\boldsymbol{k}_{G_t}(z,w)\leq-\frac{1}{2}(\log\text{\rm dist}(z,\partial G_t)+\log\text{\rm dist}(w,\partial G_t))\\+\frac{1}{2}(\log(\text{\rm dist}(z,\partial G_t)+\|z-w\|)+\log(\text{\rm dist}(w,\partial G_t)+\|z-w\|))+C,
\end{multline*}
whenever $z,w\in G_t\cap V.$\\\label{Prop2}
In particular, if $(G_j)_{j\in\mb{N}}$ is a sequence of bounded domains with $\mc{C}^2$-smooth boundaries such that $G_j$ converges to $G$, a bounded domain with $\mc{C}^2$-smooth boundary with respect to $\mc{C}^2$ topology on domains, then there exists a constant $C>0$ such that for any $j\in\mb{N}$ and any $\zeta\in G_j$ there exists a neighborhood $V=V(\zeta)$, of uniform size, with the property that
\begin{multline*}
\boldsymbol{k}_{G_j}(z,w)\leq-\frac{1}{2}(\log\text{\rm dist}(z,\partial G_j)+\log\text{\rm dist}(w,\partial G_j))\\+\frac{1}{2}(\log(\text{\rm dist}(z,\partial G_j)+\|z-w\|)+\log(\text{\rm dist}(w,\partial G_j)+\|z-w\|))+C,
\end{multline*}
whenever $z,w\in G_j\cap V.$
\end{prop}
\begin{proof}
The proof follows the lines of the proof of Proposition 19.2.9 from [\ref{JP}] (see also [\ref{FR}]) with necessary modifications. Observe that $R$ therein can be taken to be independent of $t\in T$ and $\zeta\in G_t$ (see also the proof of Proposition 9.2 from [\ref{MV}]). Also, the final constant $C$, given explicitly by $\log 2+\boldsymbol{k}_{\mb{B}(0,\frac{3}{5})\cup\mb{B}(1,\frac{3}{5})}(0,1)$ depends neither on $t\in T$ nor on $\zeta\in\partial G_t.$
\end{proof}
\begin{rem}
Observe that Proposition 9.2 from [\ref{MV}] can be deduced from Proposition \ref{Prop2}.
\end{rem}
With the aid of Theorems \ref{Main} and \ref{Main2} we can also give some uniform localization result for the Carath\'eodory-Reiffen pseudometric defined for a domain $G\s\mb{C}^n$ as
$$
\displaystyle{\boldsymbol{\gamma}_G(z;X):=\sup\left\{\big|\sum_{j=1}^n\frac{\partial f}{\partial z_j}(z)X_j\big|:f\in\mc{O}(G,\mb{D}),f(z)=0\right\} 
}
$$
for $z\in G,X=(X_1,\ldots,X_n)\in\mb{C}^n.$
\begin{prop}\label{Reiffen}
Let $(G_t)_{t\in T}$ be a family of strictly pseudoconvex domains as in Situation \ref{Situation}. Then there exists an $R>0$ such that for any $t\in T$ and any $\zeta\in\partial G_t$ the set $G_t\cap\mb{B}(\zeta, R)$ is connected and for any $X\in\mb{C}^n\setminus\{0\}$ we have
$$
\displaystyle{\lim_{G_t\cap\mb{B}(\zeta, R)\ni z\to\zeta}\frac{\boldsymbol{\gamma}_{G_t\cap\mb{B}(\zeta, R)}(z;X)}{\boldsymbol{\gamma}_{G_t}(z;X)}=1}.
$$
The convergence is uniform in $\zeta\in\partial G_t$ and $X\in\mb{C}^n\setminus\{0\}$.
\end{prop}
\begin{proof}
The proof is similar to the proof of Theorem 19.3.1 from [\ref{JP}] (see also [\ref{Gra}]). Note that thanks to Theorems \ref{Main} and \ref{Main2}, all the constants $R,\rho, L, \eta_2, \eta_1,d_2,d_1,$ and $\eta(\varepsilon)$ may be chosen independently of $t$ and $\zeta\in\partial G_t.$
\end{proof}

\end{document}